\documentclass[12pt]{amsart}

\usepackage{mathrsfs,amssymb}

\newtheorem{theorem}{Theorem}[section]
\newtheorem{lemma}[theorem]{Lemma}

\newtheorem{cor}[theorem]{Corollary}

\theoremstyle{definition}

\theoremstyle{remark}
\newtheorem{remark}[theorem]{Remark}

\numberwithin{equation}{section}

\DeclareMathOperator*{\esssup}{ess\,sup}

\DeclareMathOperator*{\essinf}{ess\,inf}

\let \la=\lambda
\let \e=\varepsilon

\let \o=\omega
\let \a=\alpha

\let \O=\Omega
\let \si=\sigma

\begin{document}
\title[estimates involving sparse operators]
{On pointwise estimates involving sparse operators}

\author{Andrei K. Lerner}
\address{Department of Mathematics,
Bar-Ilan University, 5290002 Ramat Gan, Israel}
\email{lernera@math.biu.ac.il}

\thanks{The  author was supported by the Israel Science Foundation (grant No. 953/13).}

\begin{abstract} We obtain an alternative approach to recent results by M. Lacey \cite{La} and T. Hyt\"onen {\it et al.} \cite{HRT}
about a pointwise domination of $\omega$-Calder\'on-Zygmund operators by sparse operators. This approach is rather elementary
and it also works for a class of non-integral singular operators.
\end{abstract}



\maketitle

\section{Introduction}
This paper is motivated by several recent works about a domination of
Calder\'on-Zygmund operators by sparse operators. Such a domination was first established by the author \cite{Le}
in terms of $X$-norms, where $X$ is an arbitrary Banach function space.
This result was used in order to give an alternative proof of the $A_2$ theorem obtained earlier by T.~Hyt\"onen~\cite{H}.

The $X$-norm estimate in \cite{Le} was proved for a class of $\omega$-Calder\'on-Zygmund operators with the modulus of continuity $\omega$
satisfying the logarithmic Dini condition $\int_0^1\o(t)\log\frac{1}{t}\frac{dt}{t}<\infty$.
After that, under the same assumption on $\o$, the $X$-norm bound was improved by a pointwise bound independently and simultaneously by J. Conde-Alonso and G.~Rey~\cite{CR},
and by the author and F. Nazarov~\cite{LN}.

Later, M. Lacey \cite{La} found a new method allowing him to relax the log-Dini condition in the pointwise bound till the classical Dini condition $\int_0^1\o(t)\frac{dt}{t}<\infty$. 
Very recently, T. Hyt\"onen, L. Roncal and O.~Tapiola \cite{HRT} elaborated the proof in \cite{La} to get a precise linear dependence on the Dini constant
with a subsequent application to rough singular integrals. 

In the present note we modify a main idea from Lacey's work \cite{La} with the aim to give a rather short and elementary proof of the result in~\cite{HRT}.
This yields a further simplification of the $A_2$ theorem and related bounds.
Our modification consists in a different cubic truncation of a Calder\'on-Zygmund operator $T$ with the help of an auxiliary  
``grand maximal truncated" operator ${\mathcal M}_T$. This way of truncation allows to get a very simple recursive relation for $T$, and, as a corollary, the pointwise
bound by a sparse operator.  

Notice also that our proof has an abstract nature, and, in particular, it is easily generalized to a class of singular
non-kernel operators.

\section{Main definitions}
\subsection{Sparse families and operators}
By a cube in ${\mathbb R}^n$ we mean a half-open cube $Q=\prod_{i=1}^n[a_i,a_i+h), h>0$.
Given a cube $Q_0\subset {\mathbb R}^n$, let ${\mathcal D}(Q_0)$ denote the set of all dyadic cubes with respect to $Q_0$, that is, the cubes
obtained by repeated subdivision of $Q_0$ and each of its descendants into $2^n$ congruent subcubes.

We say that a family ${\mathcal S}$ of cubes from ${\mathbb R}^n$ is $\eta$-sparse, $0<\eta<1$, if for every
$Q\in {\mathcal S}$, there exists a measurable set $E_Q\subset Q$ such that $|E_Q|\ge \eta|Q|$, and the sets
$\{E_Q\}_{Q\in {\mathcal S}}$ are pairwise disjoint. Usually $\eta$ will depend only on the dimension, and when this parameter is unessential we will skip it.

Denote $f_Q=\frac{1}{|Q|}\int_Qf$. Given a sparse family ${\mathcal S}$, define a sparse operator ${\mathcal A}_{{\mathcal S}}$ by
$${\mathcal A}_{\mathcal S}f(x)=\sum_{Q\in {\mathcal S}}f_Q\chi_{Q}(x).$$

\subsection{$\omega$-Calder\'on-Zygmund operators}
Let $\omega:[0,1]\to [0,\infty)$ be a modulus of continuity, that is, $\o$ is increasing, subadditive and $\o(0)=~0$.

We say that $T$ is an $\omega$-Calder\'on-Zygmung operator if $T$ is $L^2$ bounded, represented as
$$Tf(x)=\int_{{\mathbb R}^n}K(x,y)f(y)dy\quad\text{for all}\,\,x\not\in \text{supp}\,f$$
with kernel $K$ satisfying the size condition
$|K(x,y)|\le \frac{C_K}{|x-y|^n},x\not=y,$ and the smoothness condition
$$|K(x,y)-K(x',y)|+|K(y,x)-K(y,x')|\le \o\left(\frac{|x-x'|}{|x-y|}\right)\frac{1}{|x-y|^n}$$
for $|x-y|>2|x-x'|$.

We say that $\o$ satisfies the Dini condition if
$\|\o\|_{\text{Dini}}=\int_0^1\o(t)\frac{dt}{t}<\infty.$

\section{The Lacey-Hyt\"onen-Roncal-Tapiola Theorem}
As was mentioned in the Introduction, we give here an alternative proof of a recent result by T. Hyt\"onen {\it et al.} \cite{HRT}, which in turn is a revised version of
Lacey's domination theorem \cite{La}.

\begin{theorem}\label{mainr}
Let $T$ be an $\omega$-Calder\'on-Zygmund operator with $\omega$ satisfying the Dini condition.
Then, for every compactly supported $f\in L^1({\mathbb R}^n)$, there exists a sparse family
${\mathcal S}$ such that for a.e. $x\in {\mathbb R}^n$,
\begin{equation}\label{mainestimate}
|Tf(x)|\le c_n(\|T\|_{L^2\to L^2}+C_K+\|\omega\|_{\text{\rm{Dini}}}){\mathcal A}_{{\mathcal S}}|f|(x).
\end{equation}
\end{theorem}

We will need a number of auxiliary maximal operators.
The key role in the proof is played by the maximal operator ${\mathcal M}_T$ defined by
$${\mathcal M}_Tf(x)=\sup_{Q\ni x}\,\esssup_{\xi\in Q}|T(f\chi_{{\mathbb R}^n\setminus 3Q})(\xi)|,$$
where the supremum is taken over all cubes $Q\subset {\mathbb R}^n$ containing $x$.
This object can be called the grand maximal truncated operator. Recall that the standard maximal truncated operator
is defined by
$$T^{\star}f(x)=\sup_{\e>0}\Big|\int_{|y-x|>\e}K(x,y)f(y)dy\Big|.$$
Given a cube $Q_0$, for $x\in Q_0$ define a local version of ${\mathcal M}_T$ by
$${\mathcal M}_{T,Q_0}f(x)=\sup_{Q\ni x, Q\subset Q_0}\esssup_{\xi\in Q}|T(f\chi_{3Q_0\setminus 3Q})(\xi)|.$$
Finally, let $M$ be the standard Hardy-Littlewood maximal operator.

\begin{lemma}\label{mainl} The following pointwise estimates hold:
\begin{enumerate}
\renewcommand{\labelenumi}{(\roman{enumi})}
\item for a.e. $x\in Q_0$,
$$
|T(f\chi_{3Q_0})(x)|\le c_{n}\|T\|_{L^1\to L^{1,\infty}}|f(x)|+{\mathcal M}_{T,Q_0}f(x);
$$
\item for all $x\in {\mathbb R}^n$,
$$
{\mathcal M}_Tf(x)\le c_n(\|\omega\|_{\text{\rm{Dini}}}+C_K)Mf(x)+T^{\star}f(x).
$$
\end{enumerate}
\end{lemma}

\begin{proof} We start with part (i).
Suppose that $x\in \text{int}\,Q_0$, and let $x$ be a point of approximate continuity of $T(f\chi_{3Q_0})$ (see, e.g., \cite[p. 46]{EG}).
Then for every $\e>0$, the sets
$$E_s(x)=\{y\in B(x,s):|T(f\chi_{3Q_0})(y)-T(f\chi_{3Q_0})(x)|<\e\}$$
satisfy $\lim_{s\to 0}\frac{|E_s(x)|}{|B(x,s)|}=1$, where $B(x,s)$ is the open ball centered at $x$ of radius $s$.

Denote by $Q(x,s)$ the smallest cube centered at $x$ and containing $B(x,s)$. Let $s>0$ be so small that $Q(x,s)\subset Q_0$. Then for
a.e. $y\in E_s(x)$,
$$|T(f\chi_{3Q_0}(x)|<|T(f\chi_{3Q_0})(y)|+\e\le |T(f\chi_{3Q(x,s)})(y)|+{\mathcal M}_{T,Q_0}f(x)+\e.$$

Therefore, applying the weak type $(1,1)$ of $T$ yields
\begin{eqnarray*}
&&|T(f\chi_{3Q_0}(x)|\le \essinf_{y\in E_s(x)}|T(f\chi_{3Q(x,s)})(y)|+{\mathcal M}_{T,Q_0}f(x)+\e\\
&&\le \|T\|_{L^1\to L^{1,\infty}}\frac{1}{|E_s(x)|}\int_{3Q(x,s)}|f|+{\mathcal M}_{T,Q_0}f(x)+\e.
\end{eqnarray*}
Assuming additionally that $x$ is a Lebesgue point of $f$ and letting subsequently $s\to 0$ and $\e\to 0$, we obtain part (i).

Turn to part (ii). Let $x,\xi\in Q$. Denote by $B_x$ the closed ball centered at $x$ of radius $2\,\text{diam}\,Q$. Then $3Q\subset B_x$, and we obtain
\begin{eqnarray*}
|T(f\chi_{{\mathbb R}^n\setminus 3Q})(\xi)|&\le&
|T(f\chi_{{\mathbb R}^n\setminus B_x})(\xi)-T(f\chi_{{\mathbb R}^n\setminus B_x})(x)|\\
&+&|T(f\chi_{B_x\setminus 3Q})(\xi)|+|T(f\chi_{{\mathbb R}^n\setminus B_x})(x)|.
\end{eqnarray*}

By the smoothness condition,
\begin{eqnarray*}
&&|T(f\chi_{{\mathbb R}^n\setminus B_x})(\xi)-T(f\chi_{{\mathbb R}^n\setminus B_x})(x)|\\
&&\le \int_{|y-x|>2\,\text{diam}\,Q}|f(y)|\omega\left(\frac{\text{diam}\,Q}{|x-y|}\right)\frac{1}{|x-y|^n}dy\\
&&\le \sum_{k=1}^{\infty}\left(\frac{1}{(2^k\text{diam}\,Q)^n}\int_{2^kB_x}|f|\right)\o(2^{-k})\le c_n\|\omega\|_{\text{\rm{Dini}}}Mf(x).
\end{eqnarray*}
Next, by the size condition,
$$|T(f\chi_{B_x\setminus 3Q})(\xi)|\le c_nC_K\frac{1}{|B_x|}\int_{B_x}|f|\le c_nC_KMf(x).$$
Finally, $|T(f\chi_{{\mathbb R}^n\setminus B_x})(x)|\le T^{\star}f(x)$. Combining the obtained estimates proves part (ii).
\end{proof}

Denote $C_T=\|T\|_{L^2\to L^2}+C_K+\|\omega\|_{\text{\rm{Dini}}}$.
An examination of standard proofs (see, e.g., \cite[Ch. 8.2]{G2}) shows that
\begin{equation}\label{expr}
\max(\|T\|_{L^1\to L^{1,\infty}},\|T^{\star}\|_{L^1\to L^{1,\infty}})\le c_nC_T.
\end{equation}

\begin{proof}[Proof of Theorem \ref{mainr}]
Fix a cube $Q_0\subset {\mathbb R}^n$. Let us show that there exists a
$\frac{1}{2}$-sparse family ${\mathcal F}\subset {\mathcal D}(Q_0)$ such that for a.e. $x\in Q_0$,
\begin{equation}\label{showpoint}
|T(f\chi_{3Q_0})(x)|\le c_nC_T\sum_{Q\in {\mathcal F}}|f|_{3Q}\chi_Q(x).
\end{equation}

It suffices to prove the following recursive claim:
there exist pairwise disjoint cubes $P_j\in {\mathcal D}(Q_0)$ such that
$\sum_j|P_j|\le\frac{1}{2}|Q_0|$ and
\begin{equation}\label{recur}
|T(f\chi_{3Q_0})(x)|\chi_{Q_0}\le c_nC_T|f|_{3Q_0}+\sum_j|T(f\chi_{3P_j})|\chi_{P_j}
\end{equation}
a.e. on $Q_0$. Indeed, iterating this estimate, we immediately get (\ref{showpoint}) with ${\mathcal F}=\{P_j^k\},k\in {\mathbb Z}_+$, where
$\{P_j^0\}=\{Q_0\}$, $\{P_j^1\}=\{P_j\}$ and $\{P_j^k\}$ are the cubes obtained at the $k$-th stage of the iterative process.

Next, observe that for arbitrary pairwise disjoint cubes $P_j\in {\mathcal D}(Q_0)$,
\begin{eqnarray*}
&&|T(f\chi_{3Q_0})|\chi_{Q_0}=|T(f\chi_{3Q_0})|\chi_{Q_0\setminus \cup_jP_j}+\sum_j|T(f\chi_{3Q_0})|\chi_{P_j}\\
&&\le |T(f\chi_{3Q_0})|\chi_{Q_0\setminus \cup_jP_j}+\sum_j|T(f\chi_{3Q_0\setminus 3P_j})|\chi_{P_j}+\sum_j|T(f\chi_{3P_j})|\chi_{P_j}.
\end{eqnarray*}
Hence, in order to prove the recursive claim,
it suffices to show that one can select pairwise disjoint cubes $P_j\in {\mathcal D}(Q_0)$ with
$\sum_j|P_j|\le\frac{1}{2}|Q_0|$ and such that for a.e. $x\in Q_0$,
\begin{equation}\label{show}
|T(f\chi_{3Q_0})|\chi_{Q_0\setminus \cup_jP_j}+\sum_j|T(f\chi_{3Q_0\setminus 3P_j})|\chi_{P_j}
\le c_nC_T|f|_{3Q_0}.
\end{equation}

By part (ii) of Lemma \ref{mainl} and by (\ref{expr}), $\|{\mathcal M}_T\|_{L^1\to L^{1,\infty}}\le \a_nC_T$. Therefore, one can choose $c_n$ such that the set
$$E=\{x\in Q_0:|f|>c_n|f|_{3Q_0}\}\cup \{x\in Q_0:{\mathcal M}_{T,Q_0}f>c_nC_T|f|_{3Q_0}\}$$
will satisfy $|E|\le \frac{1}{2^{n+2}}|Q_0|$.

The Calder\'on-Zygmund decomposition applied to the function $\chi_E$ on $Q_0$ at height $\la=\frac{1}{2^{n+1}}$
produces pairwise disjoint cubes $P_j\in {\mathcal D}(Q_0)$ such that
$$\frac{1}{2^{n+1}}|P_j|\le |P_j\cap E|\le \frac{1}{2}|P_j|$$
and $|E\setminus \cup_jP_j|=0$. It follows that $\sum_j|P_j|\le \frac{1}{2}|Q_0|$ and $P_j\cap E^{c}\not=\emptyset$.
Therefore,
$$\esssup_{\xi\in P_j}|T(f\chi_{3Q_0\setminus 3P_j})(\xi)|\le c_nC_T|f|_{3Q_0}.$$
Also, by part (i) of Lemma \ref{mainl} and by (\ref{expr}), for a.e. $x\in Q_0\setminus \cup_jP_j$,
$$|T(f\chi_{3Q_0})(x)|\le c_nC_T|f|_{3Q_0},$$
which, along with the previous estimate, proves (\ref{show}) and so (\ref{showpoint}).

Take now a partition of ${\mathbb R}^n$ by cubes $R_j$ such that $\text{supp}\,(f)\subset 3R_j$ for each $j$. For example, take a cube $Q_0$ such that
$\text{supp}\,(f)\subset Q_0$ and cover $3Q_0\setminus Q_0$ by $3^n-1$ congruent cubes $R_j$. Each of them satisfies $Q_0\subset 3R_j$. Next, in the same way cover
$9Q_0\setminus 3Q_0$, and so on. The union of resulting cubes, including $Q_0$, will satisfy the desired property.

Having such a partition, apply (\ref{showpoint}) to each $R_j$. We obtain a $\frac{1}{2}$-sparse family ${\mathcal F}_j\subset {\mathcal D}(R_j)$ such that
(\ref{showpoint}) holds for a.e. $x\in R_j$ with $|Tf|$ on the left-hand side. Therefore, setting ${\mathcal F}=\cup_{j}{\mathcal F}_j$, we obtain
that ${\mathcal F}$ is $\frac{1}{2}$-sparse and for a.e. $x\in {\mathbb R}^n$,
$$|Tf(x)|\le c_nC_T\sum_{Q\in {\mathcal F}}|f|_{3Q}\chi_{Q}(x).$$
Thus, (\ref{mainestimate}) holds with a $\frac{1}{2\cdot 3^n}$-sparse family ${\mathcal S}=\{3Q: Q\in {\mathcal F}\}$.
\end{proof}

\section{Remarks and complements}
\begin{remark}\label{nonkernel}
Given a sparse family ${\mathcal S}$ and $1\le r<\infty$, define for $f\ge 0$, 
$${\mathcal A}_{r,\mathcal S}f(x)=\sum_{Q\in {\mathcal S}}\left(\frac{1}{|Q|}\int_{Q}f^r\right)^{1/r}\chi_{Q}(x).$$

Next, notice that the definition of ${\mathcal M}_T$ can be given for arbitrary (non-kernel) sublinear operator $T$. Then, the proof of Theorem \ref{mainr}
with minor modifications allows to get the following result.

\begin{theorem}\label{variant}
Assume that $T$ is of weak type $(q,q)$ and ${\mathcal M}_T$ is of weak type $(r,r)$, where $1\le q\le r<\infty$.
Then, for every compactly supported $f\in L^r({\mathbb R}^n)$, there exists a sparse family
${\mathcal S}$ such that for a.e. $x\in {\mathbb R}^n$,
\begin{equation}\label{mainestim}
|Tf(x)|\le K{\mathcal A}_{r,{\mathcal S}}|f|(x),
\end{equation}
where $K=c_{n,q,r}(\|T\|_{L^q\to L^{q,\infty}}+\|{\mathcal M}_T\|_{L^r\to L^{r,\infty}})$.
\end{theorem}

Indeed, part (i) of Lemma \ref{mainl} works with $\|T\|_{L^1\to L^{1,\infty}}$ replaced by $\|T\|_{L^q\to L^{q,\infty}}$.
Next, part (ii) of Lemma \ref{mainl} (the only part in the proof of Theorem \ref{mainr} where the kernel assumptions were used) is replaced by the
postulate that ${\mathcal M}_T$ is of weak type $(r,r)$. Finally, under trivial changes in the definition of the set $E$, we obtain that the key estimate
(\ref{recur}) holds with $c_nC_T|f|_{3Q_0}$ replaced by $K\big(\frac{1}{|3Q_0|}\int_{3Q_0}|f|^r\big)^{1/r}$. The rest of the proof is identically the same.

Recently, F. Bernicot, D. Frey and S. Petermichl \cite{BFP} obtained sharp weighted estimates for a large class of singular non-integral operators in a rather general
setting using similar ideas based on a domination by sparse operators. However, the main result in~\cite{BFP} and  Theorem \ref{variant} include some non-intersecting
cases.
\end{remark}

\begin{remark}\label{remab} It is easy to see that the cubes of the resulting  sparse family ${\mathcal S}$ in Theorem \ref{mainr} (and so in Theorem \ref{variant})
are not dyadic. But approximating an arbitrary cube by cubes from a finite number of dyadic grids
(as was shown in \cite[Lemma 2.5]{HLP} or \cite[Theorem~3.1]{LN}) yields
\begin{equation}\label{app}
{\mathcal A}_{r,\mathcal S}f(x)\le c_{n,r}\sum_{j=1}^{3^n}{\mathcal A}_{r,{\mathcal S}_j}f(x),
\end{equation}
where ${\mathcal S}_j$ is a sparse family from a dyadic grid ${\mathscr D}_j$.
\end{remark}

\begin{remark}\label{weights}
Recall that a weight (that is, a non-negative locally integrable function) $w$ satisfies the $A_p, 1<p<\infty,$ condition if
$$[w]_{A_p}=\sup_{Q}\Big(\frac{1}{|Q|}\int_Qw\,dx\Big)\Big(\frac{1}{|Q|}\int_Qw^{-\frac{1}{p-1}}\,dx\Big)^{p-1}<\infty.$$

The following lemma is well known (its variations and extensions can be found in \cite{CMP,HL,LN,M}).
\begin{lemma}\label{weight} For every $\eta$-sparse family ${\mathcal S}$ and for all $1\le r<p<\infty$,
\begin{equation}\label{shw}
\|{\mathcal A}_{r,{\mathcal S}}f\|_{L^p(w)}\le c_{n,p,r,\eta}[w]_{A_{p/r}}^{\max\big(1,\frac{1}{p-r}\big)}\|f\|_{L^p(w)}.
\end{equation}
\end{lemma}

Notice that ${\mathcal S}$ in the above mentioned works is a sparse family of dyadic cubes. Thus, the case of an arbitrary sparse family can be
treated by means of (\ref{app}). On the other hand, (\ref{app}) is not necessary for deriving Lemma \ref{weight}. In order to keep this paper essentially self-contained, we give
a proof of Lemma \ref{weight} (avoiding (\ref{app})) in the Appendix.

Theorem \ref{variant} along with Lemma \ref{weight} implies the following.

\begin{cor}\label{tweight}
Assume that $T$ is of weak type $(q,q)$ and ${\mathcal M}_T$ is of weak type $(r,r)$, where $1\le q\le r<\infty$.
Then, for all $r<p<\infty$,
$$
\|T\|_{L^p(w)}\le C[w]_{A_{p/r}}^{\max\big(1,\frac{1}{p-r}\big)},
$$
where $C=c_{n,p,q,r}(\|T\|_{L^q\to L^{q,\infty}}+\|{\mathcal M}_T\|_{L^r\to L^{r,\infty}})$.
\end{cor}
\end{remark}

\begin{remark}\label{conv}
Consider a class of rough singular integrals $Tf=\text{p.v.}f*K$, where
$K(x)=\frac{\Omega(x)}{|x|^n}$ with $\O$ homogeneous of degree zero, $\O\in L^{\infty}(S^{n-1})$ and $\int_{S^{n-1}}\O(x)d\sigma(x)=0$.

It was shown in \cite{HRT} that $\|T\|_{L^2(w)}\le c_n\|\Omega\|_{L^{\infty}}[w]_{A_2}^2$, and it was conjectured there that the squared dependence on $[w]_{A_2}$
can be replaced by the linear one. Since $T$ is of weak type $(1,1)$ (as was proved by A.~Seeger~\cite{S}), by Corollary \ref{tweight}, it would suffice to prove that ${\mathcal M}_T$ is of
weak type $(1,1)$, too. However, it is even not clear to us whether ${\mathcal M}_T$ is an $L^2$ bounded operator in this setting.
\end{remark}

\section{Appendix}
Let us prove (\ref{shw}).
Denote $\si=w^{-\frac{1}{p-1}}$ and $\nu=w^{-\frac{r}{p-r}}$. Let $E_Q$ be pairwise disjoint subsets of $Q\in {\mathcal S}$.

Since $\frac{1}{w(3Q)}\int_Qg\le \inf_QM_w^c(gw^{-1})$, where $M_w^c$ is the centered weighted maximal operator with respect to $w$,
we obtain
\begin{eqnarray*}
\sum_{Q\in{\mathcal S}}\Big(\frac{1}{w(3Q)}\int_Qg\Big)^{p'}w(E_Q)&\le& \sum_{Q\in {\mathcal S}}\int_{E_Q}M_{w}^c(gw^{-1})^{p'}w\\
&\le& \|M_w^c(gw^{-1})\|_{L^{p'}(w)}^{p'}\le c_{n,p}\|g\|_{L^{p'}(\si)}^{p'}.
\end{eqnarray*}
Similarly,
\begin{eqnarray*}
\sum_{Q\in {\mathcal S}}\Big(\frac{1}{\nu(3Q)}\int_{Q}f^r\Big)^{p/r}\nu(E_Q)\le c_{n,p,r}\|f\|_{L^p(w)}^p.
\end{eqnarray*}
Therefore, multiplying and dividing by
$$
T_{p,r}(w;Q)=\frac{w(3Q)}{w(E_Q)^{1/p'}}\frac{\nu(3Q)^{1/r}}{\nu(E_Q)^{1/p}}\frac{1}{|Q|^{1/r}}
$$
along with H\"older's inequality yield
\begin{eqnarray*}
\sum_{Q\in {\mathcal S}}\Big(\frac{1}{|Q|}\int_{Q}f^r\Big)^{1/r}\int_Qg\le c_{n,p,r}\sup_QT_{p,r}(w;Q)\|f\|_{L^p(w)}\|g\|_{L^{p'}(\si)},
\end{eqnarray*}
which, by duality, is equivalent to
$$
\|{\mathcal A}_{r,{\mathcal S}}f\|_{L^p(w)}\le c_{n,p,r}\sup_QT_{p,r}(w;Q)\|f\|_{L^p(w)}.
$$

It remains to show that $\sup_QT_{p,r}(w;Q)\le c_{n,p,r,\eta}[w]_{A_{p/r}}^{\max\big(1,\frac{1}{p-r}\big)}$.
By H\"older's inequality,
$$|Q|^{p/r}\le \eta^{-p/r}|E_Q|\le \eta^{-p/r}w(E_Q)\nu(E_Q)^{\frac{p}{r}-1}.$$
From this,
$$\frac{w(3Q)}{w(E_Q)}\Big(\frac{\nu(3Q)}{\nu(E_Q)}\Big)^{\frac{p}{r}-1}\le \eta^{-p/r}\frac{w(3Q)}{|Q|}\Big(\frac{\nu(3Q)}{|Q|}\Big)^{\frac{p}{r}-1}\le (3^n/\eta)^{p/r}[w]_{A_{p/r}},$$
and therefore,
\begin{eqnarray*}
T_{p,r}(w;Q)&=&\left[\frac{w(3Q)}{|Q|}\Big(\frac{\nu(3Q)}{|Q|}\Big)^{\frac{p}{r}-1}\right]^{1/p}\Big(\frac{w(3Q)}{w(E_Q)}\Big)^{1/p'}\Big(\frac{\nu(3Q)}{\nu(E_Q)}\Big)^{1/p}\\
&\le& 3^{n/r}[w]_{A_{p/r}}^{1/p}\left[\frac{w(3Q)}{w(E_Q)}\Big(\frac{\nu(3Q)}{\nu(E_Q)}\Big)^{\frac{p}{r}-1}\right]^{\max\big(\frac{1}{p'},\frac{r}{p(p-r)}\big)}\\
&\le& c_{n,p,r,\eta}[w]_{A_{p/r}}^{\frac{1}{p}+\max\big(\frac{1}{p'},\frac{r}{p(p-r)}\big)}=c_{n,p,r,\eta}[w]_{A_{p/r}}^{\max\big(1,\frac{1}{p-r}\big)}.
\end{eqnarray*}

\vskip 3mm
{\bf Acknowledgement.}
I am grateful to Javier Duoandikoetxea for valuable remarks on an earlier version of this paper.


\begin{thebibliography}{99}

\bibitem{BFP}
F. Bernicot, D. Frey and S. Petermichl, {\it Sharp weighted norm estimates beyond Calder\'on-Zygmund theory}, preprint.
Available at
http://arxiv.org/abs/1510.00973

\bibitem{CR}
J.M. Conde-Alonso and G. Rey, {\it
A pointwise estimate for positive dyadic shifts and some applications}, preprint.
Available at http://arxiv.org/abs/1409.4351

\bibitem{CMP}
D. Cruz-Uribe, J.M. Martell and C. P\'erez, {\it Sharp weighted estimates for classical operators}, Adv. Math.
{\bf 229} (2012), no. 1, 408-–441.

\bibitem{EG}
L.C. Evans and R.F. Gariepy, Measure theory and fine properties of functions.
Studies in Advanced Mathematics. CRC Press, Boca Raton, FL, 1992.

\bibitem{G2}
L. Grafakos, Modern Fourier analysis. Second edition. Graduate Texts in Mathematics, 250. Springer, New York, 2009.

\bibitem{H}
T.P. Hyt\"onen, {\it The sharp weighted bound for general
Calder\'on-Zygmund operators}, Annals of Math. {\bf 175} (2012), no. 3,
1473--1506.

\bibitem{HLP}
T.P. Hyt\"onen, M.T. Lacey and C. P\'erez, {\it Sharp weighted bounds for the $q$-variation of singular integrals}, Bull. Lond. Math. Soc. {\bf 45} (2013),  no. 3, 529--540.

\bibitem{HL}
T.P. Hyt\"onen and K. Li, {\it Weak and strong $A_p$-$A_{\infty}$ estimates for square functions and related operators}, preprint.
Available at http://arxiv.org/abs/1509.00273


\bibitem{HRT}
T.P. Hyt\"onen, L. Roncal and O. Tapiola, {\it
Quantitative weighted estimates for rough homogeneous singular integrals},
preprint. Available at http://arxiv.org/abs/1510.05789

\bibitem{La}
M.T. Lacey, {\it An elementary proof of the $A_2$ bound}, preprint. Available at http://arxiv.org/abs/1501.05818

\bibitem{Le}
A.K. Lerner, {\it A simple proof of the $A_2$ conjecture}, Int. Math. Res. Not. 2013, no. 14, 3159-–3170.

\bibitem{LN}
A.K. Lerner and F. Nazarov, {\it Intuitive dyadic calculus: the basics}, preprint. Available at
http://arxiv.org/abs/1508.05639

\bibitem{M}
K. Moen, {\it Sharp weighted bounds without testing or extrapolation}, Arch. Math. {\bf 99} (2012),  no. 5, 457--466.

\bibitem{S}
A. Seeger, {\it Singular integral operators with rough convolution kernels}, J. Amer. Math. Soc.  {\bf 9} (1996),  no. 1, 95--105.



\end{thebibliography}
\end{document}